\documentclass[11pt]{article}
\usepackage{amsmath, amssymb,paralist,amsthm}

\def\PC{\mathcal{P}}
\def\MC{\mathcal{M}}

\def\LC{\mathcal{L}}

\def\FC{\mathcal{F}}

\def\E{\mathbf{E}}

\def\P{\mathbf{P}}
\def\R{\mathbf{R}}

\def\1{\mathbf{1}}

\def\al{\alpha}

\newcommand{\om}{\omega}

\newcommand{\Om}{\Omega}

\def\B{\mathbf{B}}
\def\D{\mathbf{D}}
\def\E{\mathbb{E}}
\def\R{\mathbf{R}}

\def\M{\mathcal{M}}

\def\P{\mathcal{P}}

\newtheorem{prop}{Proposition}[section]
\newtheorem{theorem}{Theorem}[section]

\newtheorem{example}{Example}[section]
\newtheorem{remark}{Remark}[section]

\numberwithin{equation}{section}

\begin{document}

\centerline{\Large \bf Existence of solutions to path-dependent kinetic equations}
\smallskip
\centerline{\Large \bf and related forward - backward systems}
\bigskip
\bigskip
\centerline{\bf Vassili Kolokoltsov\footnote{Department of Statistics, University of Warwick, Coventry, CV4 7AL, UK, v.kolokoltsov@warwick.ac.uk}
\footnote{Supported by IPI RAN grants RFBR 11-01-12026 and 12-07-00115, and by the grant 4402 of the Ministry of Education and Science of Russia},  Wei Yang\footnote{Department of Mathematics and Statistics, University of Strathclyde Glasgow, G1 1XH, UK, w.yang@strath.ac.uk}}
\bigskip

\begin{abstract}
This paper is devoted to path-dependent  kinetics equations
arising, in particular, from the analysis of the coupled backward -- forward systems of equations of mean field games.
  We present local well-posedness, global existence and some regularity
results for these equations.
\end{abstract}


\section{Introduction}\label{Introduction}

For a Banach space $\B$ we denote by $\B^*$ the dual Banach space of  $\B$. The pairing between $f\in \B$ and $\mu\in \B^*$ is denoted by $(f, \mu)$.
The norm in $\B^*$ is defined by $\|\mu\|_{\B^*}=\sup_{\|f\|_{\B}\leq 1}|(f,\mu)|$. For a $T>0$, we denotes by $C([0,T], \B)$ the Banach space of continuous curves
$\eta_.: [0,T]\to B$ equipped with the norm
$\sup_{t\in[0,T]} \|\eta_t\|_{\B}$.

A deterministic dynamic in $\B^*$ can be naturally specified by a vector-valued ordinary differential equation
\begin{equation}
\label{ODE1}
\dot{\mu_t}=\Psi(t, \mu_t)
\end{equation}
with a given initial value $\mu\in \B^*$, where the mapping $(t,\eta)\mapsto \Psi(t,\eta)$ is from $\R^+\times \B^*$ to $\B^*$.
More generally, one often meets the situations when $\dot \mu$ does not belong to $\B^*$, but to some its extension. Namely,
let $\D$ be a dense subset of $\B$, which is itself a Banach space with the norm $\|\,\|_\D \geq \|\,\|_\B$. A
deterministic dynamic in $\B^*$ can be specified by equation (\ref{ODE1}), where the mapping $(t,\eta)\mapsto \Psi(t,\eta)$ is from $\R^+\times \B^*$ to $\D^*$.
Written in weak form,  equation (\ref{ODE1}) means that, for all $f\in \D$,
\begin{equation}
\label{ODE2}
(f,\dot{\mu_t})=(f,\Psi(t,\mu_t)).
\end{equation}

In many applications, equation (\ref{ODE2}) appears in the form
\begin{equation}
\label{ODE3}
\frac{d}{dt}(f,\mu_t)=(A[t,\mu_t]f,\mu_t), \quad \mu_0=\mu,
\end{equation}
where the mapping $(t,\eta)\mapsto A[t,\eta]$ is from $\R^+\times \B^*$ to bounded linear operators $A[t, \eta]: \D\mapsto \B$
 such that, for each pair $(t,\eta)\in \R^+\times \B^*$, $A[t,\eta]$ generates a strongly continuous semigroup in $\B$.
Of major interest is the case when $\B^*$ is the space of measures on a locally compact space. It turns out that, in this case and under mild 
technical assumptions, an evolution \eqref{ODE2} preserving positivity has to be of form \eqref{ODE3} with the operators $A[t, \eta]$ generating
 Feller processes, see Theorems 6.8.1 and 11.5.1 from \cite{Ko10}.

Equation \eqref{ODE3} will be referred to as the {\it general  kinetic equation}. It contains most of the basic equations from non-equilibrium statistical mechanics and evolutionary biology, see monograph \cite{Ko10} for an extensive discussion.

In this paper we are mostly interested in yet more general equation. Namely, let $\MC$ be a closed convex subset of $\B^*$, which is also closed in $\D^*$.
For a $T>0$, let $C([0,T], \M(\D^*))$ denote a closed convex subset of $C([0,T], \D^*)$ consisting of curves with values in $\M$,
and $C_{\mu}([0,T], \M(\D^*))$ a closed convex subset of $C([0,T], \M(\D^*))$, consisting of curves $\{\eta_.\}$ with initial data $\eta_0=\mu \in \MC$.

The main object of this paper is a "path-dependent" version of equation (\ref{ODE3}), that is
\begin{equation}
\label{eqkineqmeanfield}
 {d \over dt} (f, \mu_t)
 =(A[t,\{\mu_s\}_{0\leq s\leq T} ]f, \mu_t), \quad \mu_0=\mu,
\end{equation}
where $(t,\{\eta_s\}_{0\leq s\leq T})\mapsto A[t,\{\eta_s\}_{0\leq s\leq T}]$
maps $\R^+\times C_{\mu}([0,T], \M(\D^*))$ to bounded linear operators $\D\mapsto \B$.
We refer to equation \eqref{eqkineqmeanfield} as the general {\it path-dependent  kinetic equation}. It should hold for all test functions $f\in \D$.  Compared to equation \eqref{eqkineqmeanfield}, equation \eqref{ODE3} is often referred to as a {\it path-independent} case.

When the operators $A$ only depend on the {\it history} of the trajectory of $\{\mu.\}\in C_{\mu}([0,T], \M(\D^*))$,  that is
\begin{equation}
\label{''adaptive'' kinetic equation}
 {d \over dt} (f, \mu_t)
 =(A[t,\{\mu_{\leq t}\} ]f, \mu_t), \quad \mu_0=\mu,
\end{equation}
we call  \eqref{''adaptive'' kinetic equation}  an  {\it adapted kinetic equation}, where $\{\mu_{\leq t}\}$ is a shorthand for $\{\mu_{s}\}_{0\leq s\leq t}$. Adapted kinetic equations can be seen as analytic analogs of stochastic differential equations with adapted coefficients, and their well-posedness can be obtained by similar methods.
When the generators $A$ only depend on the {\it future} of the trajectory of $\{\mu.\}\in C_{\mu}([0,T], \M(\D^*))$, that is
\begin{equation}
\label{''anticipating'' kinetic equation}
 {d \over dt} (f, \mu_t)
 =(A[t,\{\mu_{\geq t}\}]f, \mu_t), \quad \mu_0=\mu,
\end{equation}
we call \eqref{''anticipating'' kinetic equation} an  {\it anticipating kinetic equation}, where $\{\mu_{\geq t}\}$ is a shorthand for $\{\mu_{s}\}_{t\leq s\leq T}$.

\begin{remark}
The terminology of {\it adaptiveness} and {\it anticipation} here should not be associated with any randomness, as in more standard usage of these words.
\end{remark}

Equation \eqref{eqkineqmeanfield} has many applications.
Let us briefly explain the crucial role played by this equation in the mean field game (MFG) methodology, which
 is based on the analysis of coupled systems of forward -- backward evolutions and which constitutes
 a quickly developing area of research in modern theory of optimization, see detail e.g. in \cite{GLL2010,HMC05,HCM3,HCM10,KLY2012}.

Assume that the objective of an agent described by a controlled stochastic process $X(s)$ (passing through $x$ at time $t$), given an evolution $\hat \mu_.$ of the empirical distributions of a large number of other players, is to maximize (over a suitable class of controls $\{u.\}$) the payoff
\[
V(t,x, \hat \mu_{\ge t}, u_{\ge t})=\E \left[ \int_t^T J(s,X(s),\hat \mu_s,u_s) \, ds +V^T (X(T))\right],
\]
 By dynamic programming the optimal payoff of such an agent
\[
 V(t,x, \hat \mu_{\ge t})=\sup_{u.}V(t,x, \hat \mu_{\ge t}, u_{\ge t})
\]
 should satisfy certain HJB equation (backward evolution). On the other hand, when all optimal controls $\{u_t=u_t(\hat \mu_{\ge t})\}$ are found, the
 empirical measure $ \mu_.$ of the resulting process satisfies the controlled kinetic
  equation of type \eqref{ODE3} (forward equation), that is
 \begin{equation}
\label{ODE3co}
\frac{d}{dt}(f,\mu_t)=(A[t,\mu_t, u_t(\hat \mu_{\ge t})]f,\mu_t), \quad \mu_0=\mu.
\end{equation}
The main consistency condition of MFG is in the requirement that the initial $\hat \mu$ coincides with the resulting $\mu$.
Equalizing $\hat \mu_.=\mu_.$ in \eqref{ODE3co} clearly leads to anticipating kinetic equation of type \eqref{''anticipating'' kinetic equation}.

Our main results concern the well-posedness of adaptive kinetic equations \eqref{''adaptive'' kinetic equation}, the local well-posedness and global existence
of anticipating and general path dependent kinetic equations and finally some regularity result for path-independent equations arising from their
probabilistic interpretations.

The rest of the paper is organised as follows. In Section \ref{mainres} our main results are formulated
and in Section \ref{secabstractnonMark} they are proved. Section \ref{secnonlinMark} yields some regularity results for the solutions of kinetic equations
leading also to simple verifiable conditions for compactness assumption \eqref{eq1thglobalexistgenkin} of our main global existence result.
Section \ref{examples} show some examples.

\section{Main results}
\label{mainres}

 Let us recall the notion of propagators needed for the proper formulation of our results.

For a set $S$, a family of mappings $U^{t,r}$ from $S$ to itself, parametrized by the pairs of numbers $r\leq t$ (resp. $t\leq r$) from a given finite or infinite interval is called a {\it (forward) propagator} (resp. a {\it backward propagator}) in $S$, if $U^{t,t}$ is the identity operator in $S$ for all $t$ and the following {\it chain rule}, or {\it propagator equation}, holds for $r\leq s\leq t$ (resp. for $t\leq s\leq r$):
$$U^{t,s}U^{s,r}=U^{t,r}.$$

A backward propagator ${U^{t,r}}$ of bounded linear operators on a Banach space $\B$ is called {\it strongly continuous} if the operators ${U^{t,r}}$ depend strongly continuously on t and r.

Suppose ${U^{t,r}}$ is a strongly continuous backward propagator of bounded linear operators on a Banach space with a common invariant domain $\D$. Let ${A_t}$, $t\geq0$, be a family of bounded linear operators $\D\mapsto \B$ that are strongly continuous in $t$ outside a set $S$ of zero-measure in $\R$. Let us say that the family ${A_t}$ {\it generates} ${U^{t,r}}$ on $\D$ if,
for any $f\in \D$, the equations
\begin{equation}
\label{Generates}
\frac{d}{ds}U^{t,s}f = U^{t,s}A_sf, \quad \frac{d}{ds}U^{s,r}f = -A_sU^{s,r}f, \quad 0\leq t\leq s\leq r,
\end{equation}
hold for all $s$ outside $S$ with the derivatives taken in the topology of $B$. In particular, if the operators $A_t$ depend strongly continuously on $t$,  equations (\ref{Generates}) hold for all $s$ and $f\in \D$, where for $s=t$ (resp. $s=r$) it is assumed to be only a right (resp. left) derivative. In the case of propagators in the space of measures, the second equation in (\ref{Generates}) is called {\it  the backward Kolomogorov equation}.

We can now formulate our main results.

\begin{theorem} [local well-posedness for general "path-dependent" case]
\label{thkineticeq}
Let $\MC$ be a bounded convex subset of $\B^*$ with $\sup_{\mu \in \MC} \|\mu \|_{\B^*} \le K$, which is closed in the norm topologies of both $\B^*$ and $\D^*$. Suppose that

(i) the linear operators $A[t,\{\xi.\}]: \D\mapsto \B$ are uniformly bounded and Lipschitz in $\{\xi.\}$, i.e. for any $\{\xi.\}, \{\eta.\}\in C_{ \mu}([0,T], \M(\D^*))$
	\begin{equation}\label{Lip_A}
		\sup_{t\in [0,T]} \|A[t, \{\xi.\}]-A[t,\{\eta.\}]\|_{\D\mapsto \B}\leq c_1\sup_{t\in[0,T]}||\xi_t-\eta_t||_{\D^*},
	\end{equation}
\begin{equation}
\label{Lip_Aa}
\sup_{t\in [0,T]} \|A[t, \{\xi.\}]\|_{\D\mapsto \B}\leq c_1
\end{equation}
for a positive constant $c_1$;

(ii) for any $\{\xi_.\}\in C_{\mu}([0,T],\M(\D^*))$, let the operator curve $A[t,\{\xi_.\}]:\D\mapsto \B$ generate a strongly continuous backward propagator of bounded linear operators $U^{t,s}[\{\xi_.\}]$ in $\B$, $0 \leq t \leq s$, on the common invariant domain $\D$, such that
	\begin{equation}\label{BDD}
		||U^{t,s}[\{\xi.\}]||_{\D\mapsto \D}\leq c_2\,\, \text{and}\,\,||U^{t,s}[\{\xi.\}]||_{\B\mapsto \B}\leq c_3, \quad t\leq s,
	\end{equation}
for some positive constants $c_2,c_3$, and with their dual propagators $\tilde{U}^{s,t}[\{\xi_.\}]$ preserving the set $\M$.

Then, if
\begin{equation}
\label{eq0thkineticeq}
c_1c_2c_3 K T< 1,
\end{equation}
 the Cauchy problem
\begin{equation}
\label{weak nonlinear Cauchy problem}
\frac{d}{dt}(f,\mu_t)=(A[t,\{\mu.\}]f, \mu_t),\quad \mu_0=\mu, \,\,t\in [0,T],
\end{equation}
is  well posed, that is for any $\mu\in \M$, it has a unique solution $\Phi^t(\mu) \in \M$ (that is (\ref{weak nonlinear Cauchy problem}) holds for all $f\in \D$) that depends Lipschitz continuously on time $t$ and the initial data in the norm of $\D^*$, i.e.
\begin{equation}
\label{eq2thkineticeq}
\|\Phi^s(\mu)-\Phi^t (\mu)\|_{\D^*}\leq  c_1 c_2 (s-t), \quad 0\le t \le s \le T,
\end{equation}
and for $\mu,\eta\in\M$
\begin{equation}
\label{eq1thkineticeq}
\begin{split}
\|\{\Phi^.(\mu)\}-\{\Phi^.(\eta)\}\|_{C([0,T],\M(\D^*))}
&\leq \frac{c_2}{1-c_1c_2c_3KT}\|\mu-\eta\|_{\D^*}.
\end{split}
\end{equation}
\end{theorem}

\begin{theorem}(global well-posedness for general "path-independent" case)
\label{propkineticeqnonlMark}
Under the assumptions in Theorem \ref{thkineticeq}, but without the locality constraint
\eqref{eq0thkineticeq},
the Cauchy problem for kinetic equation
\begin{equation}
\label{eq1propkineticeqnonlMark}
\frac{d}{dt}(f,\mu_t)=(A[t,\mu_t]f,\mu_t),\quad \mu_s=\mu
\end{equation}
is well-posed, i.e. for any $\mu \in \MC$, $s\in [0,T]$,  it has a
unique solution $\tilde U^{t,s}(\mu) \in \MC$, $t\in [s,T]$, and the transformations $\tilde U^{t,s}$ of
$\MC$ form a propagator depending Lipschitz
continuously on time $t$ and the initial data in the norm of
$\D^*$, i.e.
\begin{equation}
\label{eqLipcontnonlinearprop}
 \|\tilde U^{t,s}(\mu)-\tilde U^{t,s}(\eta)\|_{\D^*} \le c(T, K) \|\mu-\eta \|_{\D^*},
\end{equation}
with a constant $c(T,K)$ depending on $T$ and $K$.
\end{theorem}

\begin{theorem}(global wellposedness for an "adapted" case)
\label{thkineticeqnonanticipating}
Under the assumptions in Theorem \ref{thkineticeq}, but without the locality constraint
\eqref{eq0thkineticeq},
the Cauchy problem
\begin{equation}
\label{weak nonlinear Cauchy problem global NE}
\frac{d}{dt}(f,\mu_t)=(A[t,\{\mu_{\leq t}\}]f, \mu_t),\quad \mu_0=\mu,
\end{equation}
is well posed in $\M$ and its unique solution depends Lipschitz continuously on initial data in the norm of $\D^*$.
\end{theorem}

\begin{theorem}[global existence of the solution for general "path dependent" case]
\label{thglobalexistgenkin}
Under the assumptions in Theorem \ref{thkineticeq}, but without the locality constraint
\eqref{eq0thkineticeq}, assume additionally that for any $t$ from a dense subset of $[0,T]$, the set
\begin{equation}
\label{eq1thglobalexistgenkin}
\{\tilde{U}^{t,0}[\{\xi.\}]\mu:\,\,\{\xi.\} \in C_\mu([0,T],\MC(\D^*))\}
\end{equation}
 is relatively compact in $\M$. Then a solution to the Cauchy problem
\begin{equation}
\label{weak nonlinear Cauchy problem-GE}
\frac{d}{dt}(f,\mu_t)=(A[t,\{\mu_s\}_{s\in[0,T]}]f, \mu_t),\quad \mu_0=\mu, t\geq 0,
\end{equation}
exists in $\M$.
\end{theorem}

In Proposition \ref{propmomentcomp} in Section \ref{secnonlinMark}, we give the conditions under which the compactness assumption \eqref{eq1thglobalexistgenkin} holds.

\section{Proofs of the main results}
\label{secabstractnonMark}

{\bf Proof of Theorem \ref{thkineticeq} }

By duality,  for any $\{\xi_.^1\}, \{\xi_.^2\}\in C_{\mu}([0,T],\M(\D^*))$
\begin{equation*}
\begin{split}
(f,(\tilde{U}^{t,0}[\{\xi_.^1\}]- \tilde{U}^{t,0}[\{\xi_.^2\}])\mu)
=((U^{0,t}[\{\xi_.^1\}]- U^{0,t}[\{\xi_.^2\}])f,\mu).
\end{split}
\end{equation*}

Next, we need to estimate the difference of the two propagators. Define an operator-valued function $Y(r):=U^{0,r}[\{\xi_.^2\}]    U^{r,t}[\{\xi_.^1\}]$. Since $U^{t,t}[\{\xi_.^i\}], i=1,2,$ are identity operators, $Y(t)=U^{0,t}[\{\xi_.^2\}] $ and $Y(0)=U^{0,t}[\{\xi_.^1\}] $. By (\ref{Generates}), we get
\begin{equation*}
\begin{split}
&\hspace{.5cm} U^{0,t}[\{\xi_.^2\}]-U^{0,t}[\{\xi_.^1\}]=U^{0,r}[\{\xi_.^2\}]   U^{r,t}[\{\xi_.^1\}]\big|_{r=0}^t\\
&=\int_0^t\frac{d}{dr} \left(U^{0,r}[\{\xi_.^2\}]   U^{r,t}[\{\xi_.^1\}] \right) dr\\
&=\int_0^t \Big( U^{0,r}[\{\xi_.^2\}] A[r,\{\xi_.^2\}] U^{r,t}[\{\xi_.^1\}] -U^{0,r}[\{\xi_.^2\}] A[r,\{\xi_.^1\}] U^{r,t}[\{\xi_.^1\}] \Big)ds\\
&=\int_0^tU^{0,r}[\{\xi_.^2\}](A[r,\{\xi.^2\}]-A[r,\{\xi.^1\}])U^{r,t}[\{\xi_.^1\}] ds.
\end{split}
\end{equation*}
Then, together with assumptions \eqref{Lip_A} and \eqref{BDD},
\begin{equation}
\label{Lip U}
\begin{split}
&||(\tilde{U}^{t,0}[\{\xi_.^1\}]- \tilde{U}^{t,0}[\{\xi_.^2\}])\mu)||_{\D^*}\\
\leq &||U^{0,t}[\{\xi_.^1\}]- U^{0,t}[\{\xi_.^2\}]||_{\D\mapsto \B} ||\mu||_{\B^*}\\
\leq & c_1c_2c_3 tK \|\{\xi_.^1\}-\{\xi_.^2\}\|_{C([0,T],\M(\D^*))}.
\end{split}
\end{equation}
Consequently, if \eqref{eq0thkineticeq} holds, the mapping $\{\xi_.\}\mapsto \{\tilde{U}^{t,0}[\{\xi_.\}]\}_{t\in [0,T]}$ is a contraction in $C_{\mu}([0,T],\M(\D^*))$. Hence by the contraction principle there exists a unique fixed point for this mapping and hence a unique solution to equation \eqref{weak nonlinear Cauchy problem}.

Inequality \eqref{eq2thkineticeq} follows directly from \eqref{weak nonlinear Cauchy problem}.
Finally, if $\Phi^t(\mu) = \mu_t$ and $\Phi^t(\eta) = \eta_t$, then
\begin{equation*}
\begin{split}
\mu_t-\eta_t &= \tilde{U}^{t,0}[\{\mu.\}]\mu-\tilde{U}^{t,0}[\{\eta_.\}]\eta\\
 &= (\tilde{U}^{t,0}[\{\mu.\}]-\tilde{U}^{t,0}[\{\eta.\}])\mu
  +\tilde{U}^{t,0}[\{\eta.\}](\mu-\eta).
\end{split}
\end{equation*}
From \eqref{BDD} and \eqref{Lip U},
\begin{equation}
\label{eq3thkineticeq}
\|\{\mu_.\}-\{\eta_.\}\|_{C([0,T],\M(\D^*))} \leq  c_1c_2c_3 T K \|\{\mu_.\}-\{\eta_.\}\|_{C([0,T],\M(\D^*)))} + c_2\|\mu-\eta\|_{\D^*}
\end{equation}
implying \eqref{eq1thkineticeq}.
 \qed\\[.3em]

{\bf Proof of Theorem \ref{propkineticeqnonlMark} }
The global unique solution of \eqref{eq1propkineticeqnonlMark} is constructed by extending local unique solutions of \eqref{weak nonlinear Cauchy problem} via iterations, as is routinely performed in the theory of ordinary differential equations (ODE).
 \qed\\[.3em]

{\bf Proof of Theorem \ref{thkineticeqnonanticipating} }

For a $\mu \in \MC$, let us construct an approximating sequence $\{\xi^n_.\} \in C([0,T],\MC(\D^*))$,
 $n=0,1,\cdots$, by defining $\xi_t^0=\mu$ for $t\in [0,T]$  and then recursively
\[
\xi^n_t=\tilde {U}^{t,0}[\xi^{n-1}_{\leq t}]\mu, \quad \forall t\in[0,T].
\]
By non-anticipation, arguing as in the proof of \eqref{Lip U} above, we first get the estimate
\[
\sup_{0\leq r\leq t}||\xi^1_{r}-\xi^0_{r}||_{\D^*}\leq c_1c_2c_3Kt,
\]
and then recursively
\[
\sup_{0\leq r\leq t}||\xi^{n}_{r}-\xi^{n-1}_{r}||_{\D^*}\leq c_1c_2c_3K
\int_0^t \sup_{0\leq r\leq s}||\xi^{n-1}_{ r}-\xi^{n-2}_{ r}||_{\D^*}\,ds
\]
that implies (by straightforward induction) that, for all $t\in [0,T]$,
\[
\|\{\xi^{n}_.\}-\{\xi^{n-1}_.\}\|_{C([0,t],\M(\D^*))}
=\sup_{0\leq r\leq t}||\xi^{n}_{r}-\xi^{n-1}_{r}||_{\D^*} \le \frac{1}{n!}(c_1c_2c_3 K t)^n.
\]
Hence, the partial sums on the r.h.s. of the obvious equation
\[
\xi^n_{\leq t}=(\xi^n_{\leq t}-\xi^{n-1}_{\leq t})+\cdots+(\xi^{1}_{\leq t}-\xi^{0}_{\leq t})+\xi^{0}_{\leq t}
\]
converge, and thus the sequence $\xi_{\cdot}^n$ converges in $C([0,T],\D^*)$. The limit is clearly a solution to
\eqref{weak nonlinear Cauchy problem global NE}.

To prove uniqueness and continuous dependence on the initial condition,
let us assume that $\mu_t$ and $\eta_t$ are some solutions with the initial conditions
$\mu$ and $\eta$ respectively. Instead of \eqref{eq3thkineticeq}, we now get
\[
\|\{\mu_.\}-\{\eta_.\}\|_{C([0,t],\M(\D^*))} \leq  c_1c_2c_3 K \int_0^t \|\{\mu_.\}-\{\eta_.\}\|_{C([0,s],\M(\D^*))}\, ds
 + c_2\|\mu-\eta\|_{\D^*}.
\]
By Gronwall's lemma, this implies
\[
\|\{\mu_.\}-\{\eta_.\}\|_{C([0,t],\M(\D^*))} \le c_2\|\mu-\eta\|_{\D^*} e^{c_1c_2c_3 Kt}, \quad t\in[0,T]
\]
yielding uniqueness and Lipchitz continuity of solutions with respect to initial data.
\qed
\\[.3em]

{\bf Proof of Theorem \ref{thglobalexistgenkin}}

Since $\M$ is convex, the space $C_{\mu}([0,T], \M(\D^*))$ is also convex. Since the dual operators $\tilde{U}^{t,0}[\{\xi.\}]$ preserve the set $\M$, for any $\{\xi.\} \in  C_{ \mu}([0,T], \M(\D^*))$, the curve $\tilde{U}^{t,0}[\{\xi.\}]\mu$ belongs to $C_{ \mu}([0,T], \M(\D^*))$ as a function of $t$.
Hence, the mapping $\{\xi.\}\to \{\tilde{U}^{t,0}[\{\xi.\}]\mu, t\in [0,T]\}$ is from $C_{ \mu}([0,T], \M(\D^*))$ to itself. Moreover, by (\ref{Lip U}), this mapping is Lipschitz continuous.

Denote $\hat{C}=\{ \{\tilde{U}^{\cdot,0}[\{\xi.\}]\mu\}:  \{\xi.\}\in C_{\mu}([0,T], \M(\D^*))\}$. Together with (\ref{eq2thkineticeq}), the assumption that set \eqref{eq1thglobalexistgenkin} is compact in $\M$ for any $t$ from a dense subset of $[0,T]$ implies that the set $\hat{C}$
 is relatively compact in $C_{\mu}([0,T], \M(\D^*))\}$ (Arzela-Ascoli Theorem, see A.21 in \cite{K1997}).

Finally, by Schauder fixed point theorem, there exists a fixed point in $\hat C\subset C_{\mu}([0,T], \M(\D^*))$, which gives the existence of a solution to (\ref{weak nonlinear Cauchy problem-GE}).
\qed

\section{Nonlinear Markov evolutions and its regularity}
\label{secnonlinMark}

This section is designed to provide a probabilistic interpretation and, as a consequence, certain regularity properties for nonlinear Markov evolution $\mu_t$ solving kinetic equation (\ref{eq1propkineticeqnonlMark}) in the case when $\B=C_{\infty}(\R^d)$ and $\M=\P(\R^d)$ is the set of probability measures on $\R^d$,
 so that $\B^*$ is the space of signed Borel measures on $\R^d$ and $K=\sup_{\mu \in \P(\R^d)} \|\mu \|_{\B^*}=1$. As a consequence, w shall present a simple 
 criterion for the main compactness assumption of Theorem \ref{thglobalexistgenkin}.  

We will use the following notations.

$C_{Lip}(\mathbf{R}^d)$ is the Banach space of bounded Lipschitz continuous functions $f$ on $\mathbf{R}^d$ with the norm $\|f\|_{{Lip}}=\sup_x|f(x)|+\sup_{x\neq y}\frac{|f(x)-f(y)|}{|x-y|}$.

$C_{\infty}(\mathbf{R}^d)$ is the Banach space of bounded continuous functions $f$ on $\R^d$ with $\lim_{x\rightarrow\infty}f(x)=0$, equipped with sup-norm $\|f\|_{C_{\infty}(\mathbf{R}^d)}:=\sup_x|f(x)|$.

$C_{\infty}^1(\mathbf{R}^d)$ is the Banach space of continuously differentiable and bounded functions $f$ on $\mathbf{R}^d$  such that the derivative $f'$ belongs to $C_{\infty}(\R^d)$, equipped with the norm $\|f\|_{C_{\infty}^1(\mathbf{R}^d)}:=\sup_x|f'(x)|$.

$C_{\infty}^2(\mathbf{R}^d)$ is the Banach space of twice continuously differentiable and bounded functions $f$ on $\mathbf{R}^d$  such that the first derivative $f'$ and the second derivative $f''$ belong to $C_{\infty}(\R^d)$, equipped with the norm $\|f\|_{C_{\infty}^2(\mathbf{R}^d)}:=\sup_x(|f'(x)|+|f''(x)|)$.

Let $\{A[t,\mu]: t\geq 0, \mu\in \P(\R^d)\}$ be a family of operators in $C_{\infty}(\R^d)$ of the L\'evy-Khintchin type, that is
\begin{equation}
\label{fellergenerator 3}
\begin{split}
&A[t,\mu]f(z)=\frac{1}{2}(G(t,z,\mu)\nabla,\nabla)f(z)+ (b(t,z,\mu),\nabla f(z))\\
 &+\int (f(z+y)-f(z)-(\nabla f (z), y){\bf 1}_{B_1}(y))\nu (t,z,\mu,dy),
\end{split}
\end{equation}
where $\nabla$ denotes the gradient operator; for $(t,z,\mu)\in [0,T]\times\R^d\times\P(\R^d)$, $G(t,z,\mu)$ is a symmetric non-negative matrix, $b(t,z,\mu)$ is a vector, $\nu(t,z,\mu,\cdot)$ is a L\'evy measure on $\R^d$, i.e.
\begin{equation}
\label{condlevy0}
\int_{\R^d} \min (1,|y|^2)\nu(t, z, \mu, dy) <\infty, \quad \nu (t, z, \mu, \{0\})=0,
\end{equation}
depending measurably on $t, z, \mu$, and ${\bf 1}_{B_1}$ denotes, as usual, the indicator function of the unit ball in $\R^d$.
Assume that each operator \eqref{fellergenerator 3} generates a Feller process with one and the same domain $\D$ such that $C^2_{\infty}(\R^d) \subseteq \D\subseteq C^1_{\infty}(\R^d)$.

\begin{prop}\label{Prop31}
 Suppose the assumptions of Theorem \ref{propkineticeqnonlMark} are fulfilled with generators $A[t,\mu]$ of type (\ref{fellergenerator 3})
 and a probability measure $\mu$ is given.
 Then there exists a family of processes
 $\{X_{s,t}^{\mu}: \mu\in \P(\R^d)\}$
defined on a certain filtered probability space $(\Om, \FC, \{\FC_t\},\P(\R^d))$
such that $\mu_t=\LC(X_{s,t}^{\mu})$ solves the
Cauchy problem for equation \eqref{eq1propkineticeqnonlMark} with initial condition $\mu$
and $\{X_{s,t}^{\mu}\}$ solves the nonlinear martingale problem, specified by the family $\{A[t,\mu]\}$, that is, for any $f\in \D$,
\begin{equation}
\label{eqdefNMP}
f(X_{s,t}^{\mu})-\int^t_s A[\tau, \LC(X_{s,\tau}^{\mu})]f(X_{s,\tau}^{\mu})d\tau, \quad s\leq t
\end{equation}
is a martingale.
\end{prop}

\proof
 By the assumptions of Theorem \ref{propkineticeqnonlMark}, a solution $\mu_t\in\P(\mathbf{R}^d)$
 of equation \eqref{eq1propkineticeqnonlMark} with initial condition $\mu_s=\mu$ specifies a propagator $\tilde{U}^{t,r}[\mu_.]$, $s\leq r\leq t$,  of linear transformations in $\B^*$,
 solving the Cauchy problems for equation
\begin{equation}\label{ODE4}
	\frac{d}{dt}(f,\nu_t)=(A[t, \mu_t]f,\nu_t).
\end{equation}
In its turn, for any $\nu \in \P(\R^d)$, equation \eqref{ODE4} specifies marginal distributions of a usual (linear) Markov process $\{X_{s,t}^{\mu}(\nu)\}$ in $\mathbf{R}^d$
with the initial measure $\nu$. Clearly, the process $\{X_{s,t}^{\mu}(\mu)\}$ is a solution to our martingale problem.
\qed

We shall refer to the family of processes constructed in Proposition \ref{Prop31} as to {\it nonlinear Markov process} generated by the family $A[t,\mu]$.

Using martingales allows us to prove the following useful regularity property for the solution of kinetic equations.
\begin{prop}
\label{proppmomentsfornonlin}
Suppose the assumptions of Theorem \ref{propkineticeqnonlMark} are fulfilled for a kinetic equation of "path-independent" type \eqref{eq1propkineticeqnonlMark} with generators
$A[t,\mu]$ of type \eqref{fellergenerator 3}. Let $\{X_{s,t}^\mu\}$ denote a nonlinear Markov process constructed from the family of generators $A[t,\mu]$ by Proposition \ref{Prop31}.
Assume, for $p \in (0,2]$ and $P>0$, the following boundedness condition holds:
\begin{equation}
\label{EBDD}
\sup_{x\in R^d,\,t\geq0, \,\mu\in\M}\max\big\{|G(t,x,\mu)|,|b(t,x,\mu)|,
\int \min (|y|^2, |y|^p) \nu(t,x,\mu,dy)\big\} \le P,
\end{equation}
and the initial measure $\mu_s=\mu$ has a finite $p$th order moment, i.e.
\[
\int |x|^p \mu (dx)=p_{\mu} <\infty.
\]

Then the distributions $\LC(X_{s,t}^{\mu}) =\Phi^{t,s}(\mu)$, solving the Cauchy problem for equation \eqref{eq1propkineticeqnonlMark} with initial condition $\mu_s$ have uniformly bounded $p$th moments, i.e.
\begin{equation}
\label{eq1proppmomentsfornonlin}
\int |x|^p \Phi^{t,s}(\mu) (dx) \le c(T,P)[1+p_{\mu}],
\end{equation}
and are
$\frac{1}{2}$-H\"older continuous with respect to $t$ in the space $(C_{Lip}(\R^d))^*$, i.e.
\begin{equation}
\label{Holder}
||\Phi^{t_1,s}(\mu)-\Phi^{t_2,s}(\mu)||_{(C_{Lip}(\R^d))^*}\leq c(T,P) \sqrt{|t_1-t_2|},
\quad \forall t_1, t_2\geq s \geq 0,
\end{equation}
with a positive constant $c$.
\end{prop}

\proof
For a fixed trajectory $\{\mu_t\}_{t\geq 0}$ with initial value $\mu$, one can
consider $\{X_{s,t}^\mu\}$ as a usual Markov process.
Using the estimates for the moments of such processes from formula (5.61) of
\cite{VK2} (more precisely, its straightforward extension to time non-homogeneous case),
one obtains from \eqref{EBDD} that
\begin{equation}
\label{eq2proppmomentsfornonlin}
\E \left[\min \big(|X_{s,t}^{\mu}-\hat x|^2, |X_{s,t}^{\mu}-\hat x|^p\big)|X_{s,s}^{\mu}=\hat x)\right]
\le e^{C(T,P)(t-s)}-1.
\end{equation}
This implies \eqref{eq1proppmomentsfornonlin}.
Moreover, \eqref{eq2proppmomentsfornonlin} implies that
 \begin{equation}
\label{eq3proppmomentsfornonlin}
\E \left[|X_{s,t}^{\mu}-\hat x| {\bf 1}_{|X_{s,t}^{\mu}-\hat x| \le 1}|X_{s,s}^{\mu}=\hat x\right[
\le (e^{C(T,P)(t-s)}-1)^{1/2}\le C(T,P) \sqrt {t-s}
\end{equation}
 \begin{equation}
\label{eq4proppmomentsfornonlin}
\E \left[|X_{s,t}^{\mu}-\hat x| {\bf 1}_{|X_{s,t}^{\mu}-\hat x| \ge 1}|X_{s,s}^{\mu}=\hat x\right]
\le (e^{C(T,P)(t-s)}-1)^{1/p} \le C(T,P) (t-s)^{1/p},
\end{equation}
and consequently
 \begin{equation}
\label{eq5proppmomentsfornonlin}
\E \left(|X_{s,t}^{\mu}-\hat x| \, |X_{s,s}^{\mu}=\hat x\right)
 \le C(T,P) \sqrt {t-s},
\end{equation}
where constants $C(T,P)$ can have different values in various formulas above.

Since $\Phi^{t,s}(\mu)$ is the distribution law of the process $\{X_{s,t}^\mu\}$,
\begin{equation}
\label{eq6proppmomentsfornonlin}
\begin{split}
||\Phi^{s,t_1}(\mu)-\Phi^{s,t_2}(\mu)||_{(C_{Lip}(\R^d))^*}
= & \sup_{||\psi||_{C_{Lip}(\R^d)}\leq 1}   \big| \E \psi(X_{s,t_1}^\mu) - \E\psi(X_{s,t_2}^\mu) \big|  \\
\leq & \sup_{||\psi||_{C_{Lip}(\R^d)}\leq 1} \E \big|\psi(X_{s,t_1}^\mu)-\psi(X_{s,t_2}^\mu)\big| \\
\leq & \E \big|X_{s,t_1}^\mu-X_{s,t_2}^\mu\big|.
\end{split}
\end{equation}
From \eqref{eq5proppmomentsfornonlin}, \eqref{eq6proppmomentsfornonlin} and Markov property, we get
\eqref{Holder} as required.
\qed

\begin{remark} For the case of diffusions with $G=1$ , H\"older continuity \eqref{Holder} was proved in \cite{HCM3}.
\end{remark}

Our main purpose for presenting Proposition \ref{proppmomentsfornonlin} lies in the following corollary.

\begin{prop}
 \label{propmomentcomp}
 Under the assumptions of Theorem \ref{thkineticeq} for generators
$A[t,\{\mu_.\}]$ of L\'evy-Khintchin type \eqref{fellergenerator 3}, but without locality condition
\eqref{eq0thkineticeq}, suppose the boundedness condition \eqref{EBDD} holds for some $p \in (0,2]$ and $P>0$.
Then the compactness condition from Theorem \ref{thglobalexistgenkin} (stating that set
\eqref{eq1thglobalexistgenkin} is compact in $\PC(\R^d)$) holds for any initial measure $\mu$
with a finite moment of $p$th order.
 \end{prop}

 \begin{proof}
 It follows from  \eqref{eq1proppmomentsfornonlin} and an observation that a set of probability laws on $\R^d$ with a bounded $p$th moment, $p>0$, is tight and hence relatively compact.
 \end{proof}

\section{Basic examples of operators $A[t,\mu]$}
\label{examples}

In this section, we present some basic examples of generators that fit to assumptions of our main Theorems and are relevant to the study of mean field games.

Notice that the most nontrivial condition of Theorem \ref{thkineticeq} is (ii), as it concerns the difficult question from the theory of usual Markov process, on when
a given pre-generator of L\'evy-Khintchine type does really generate a Markov process.
Even more difficult is the situation with time-dependent generators, as the standard semigroup methods
(resolvents and Hille-Phillips-Iosida theorem) are not applicable.

\begin{example} Nonlinear L\'evy processes are specified by a families of generators of type (\ref{fellergenerator 3}) such that all coefficients do not depend on $z$, i.e.
\begin{equation}
\begin{split}
A[t,\mu]f(x)=&\frac{1}{2}(G(t,\mu) \nabla,\nabla)f(x)+(b(t,\mu), \nabla f)(x)\\
&+\int [f(x+y)-f(x)-(y, \nabla f(x))\mathbf{1}_{B_1}(y)]\nu(t,\mu,dy). \nonumber
\end{split}
\end{equation}

The following statement is a consequence of Proposition 7.1 from \cite{Ko10}.

\begin{prop}
\label{NLP}
Supposed that the coefficients $G,b,\nu$ are continuous in $t$ and Lipschitz continuous in $\mu$ in the norm of Banach space $(C^2_{\infty}(\R^d))^*$, i.e.
\begin{equation*}
\begin{split}
\|G(t,\mu)-G(t,\eta)\|+&\|b(t,\mu)-b(t,\eta)\|+\int\min(1,|y|^2)|\nu(t,\mu,dy)-\nu(t,\eta, dy)|\\
\leq &c||\mu-\eta||_{(C^2_{\infty}(\R^d))^*},\quad \forall t\geq 0,
\end{split}
\end{equation*}
with a positive constant $c$, then condition (ii) of Theorem \ref{thkineticeq} holds with $\D =C^2_{\infty} (\R^d)$.
\end{prop}
\end{example}

Notice that the most natural examples of a functional $F$ on measures that are Lipschitz continuous (or even smooth) in space $(C^2_{\infty}(\R^d))^*$
are supplied by smooth functions of monomials $\int g(x_1, \cdots , x_n) \mu(dx_1) \cdots \mu (dx_n)$ with sufficient smooth functions $g$.

\begin{example}
\label{MV Diffusion}
McKean-Vlasov diffusion are specified by the following stochastic differential equation
\[
dX_t = b(t, X_t, \mu_t)dt + \sigma(t, X_t, \mu_t)dW_t,
\]
where drift coefficient $b:\R^+\times \mathbf{R}^d \times \P(\mathbf{R}^d)\rightarrow \mathbf{R}^d$, diffusion coefficient $\sigma: \R^+\times\mathbf{R}^d \times \P(\mathbf{R}^d)\rightarrow \mathbf{R}^d$ and $W_t$ is a standard Brownian motion. The corresponding generator is given by
\[
A[t,\mu]f(x) = (b(t, x,\mu), \nabla f(x))+ \frac{1}{2}(G(t, x,\mu), \nabla^2f(x)),\quad f\in C^2_{\infty}(\R^d),
\]
where $G(t, x,\mu)=tr\{\sigma (t, x,\mu)\sigma^T (t, x,\mu)\}$.
It is well known (and follows from Ito's calculus) that if the coefficients of a diffusion are
Lipshitz continuous, the corresponding SDE is well posed, implying the following.

\begin{prop}
\label{propnonlindif}
If $G, b$ are continuous in $t$, Lipshitz continuous in $x$ and Lipschitz continuous in $\mu$ in the topology of $(C^2_{\infty}(\R^d))^*$,
then the condition (ii) of Theorem \ref{thkineticeq} is satisfied.
\end{prop}
\end{example}

\begin{example} Nonlinear stable-like processes (including tempered ones) are specified by the families
\begin{equation}
\begin{split}
A[t,\mu]f(x) = &(b(t,x,\mu), \nabla f(x))+\int (f(x+y)-f(x)) \nu(t,x,\mu,dy) \\
&+ \int_0^K d|y| \int_{S^{d-1}}
 a(t,x,s)\frac {f(x+y)-f(x)-(y,\nabla f(x))}{|y|^{\al_t(x,s)+1}}
 \om_t(ds).
\end{split}
\end{equation}
Here $s=y/|y|$, $K>0$, $\om_{t}$ are certain finite Borel measures on
$S^{d-1}$ and $\nu(t,x,\mu,dy)$ are finite measures, $a$, $\al$ are positive bounded functions with
$\al \in (0,2)$.

The following result is a corollary of (a straightforward time-nonhomogeneous extension of) Proposition
4.6.2 of \cite{VK2}.
\begin{prop}
\label{propstablelikesmooth}
If all coefficients are continuous in $t$, $a,\al$ are $C^1$-functions in $x,s$, $b$ and $\nu$ are Lipshitz continuous in $x$ and $\mu$ (with $\mu$ taken in the topology $(C^2_{\infty}(\R^d))^*$,
then all conditions of Theorem \ref{thkineticeq} are satisfied with $\D =C^2_{\infty} (\R^d)$.
\end{prop}
\end{example}

\begin{example}
 \label{OAMO}
Processes of order at most one are specified by the families
\begin{equation}\label{eqshortgeneratornonlin}
 A[t,\mu]f(x)=(b(t,x,\mu),\nabla f(x))+\int_{\mathbf{R}^d}(f(x+y)-f(x))\nu (t,x,\mu,dy), \nonumber
\end{equation}
with the L\'evy measures $\nu$ having finite first moment $\int |y|\nu(t,x,\mu,dy)$.
The next result is established in Theorem 4.17 of \cite{Ko10}.

\begin{prop}
If $b,\nu$ are continuous in $t$ and Lipschitz continuous in $\mu$, i.e.
\begin{equation*}
\begin{split}
||b(t,x,\mu)-b(t,x,\eta)||+&\int |y|\nu(t,x,\mu,dy)-\nu(t,x,\eta, dy)|\\
\leq &c||\mu-\eta||_{(C^1_{\infty}(\R^d))^*},\quad \quad \forall t\geq 0, x\in \R^d
\end{split}
\end{equation*}
and Lipshitz continuous in $x$,
then condition (ii) of Theorem \ref{thkineticeq} is satisfied with $\D =C^1_{\infty} (\R^d)$.
\end{prop}
\end{example}

The generators of order at most one describe a variety of models including spatially homogeneous and
mollified Boltzmann equation and interacting $\al$-stable laws with $\al <1$.

\begin{example}
\label{integralform} Pure jump processes are specified by integral generators of the form
\begin{equation}
A[t,\mu]f(x) = \int_{\R^d}(f(y)-f(x))\nu(t, x,\mu,dy).
\end{equation}
If measures $\nu (t, x,\mu_t, u, dy)$ are uniformly bounded,
the conditions of Theorem \ref{thkineticeq} are satisfied with $\D =C_{\infty} (\R^d)$.
For unbounded rates we refer to \cite{Ko10} (and references therein) for a detailed discussion.
\end{example}

Let us note finally that not all interesting evolution of type \eqref{ODE3} satisfy the Lipschitz continuity assumption used in our main results.
For instance, a different type of continuity should be applied for coefficients depending on measures via their quantiles, e.g. value at risk (VAR).
This type of evolution is analyzed in \cite{K2012} inspired by preprint \cite{CKL}.

\end{document}